\documentclass[12pt]{amsart}

\usepackage{amsmath,amssymb,amsfonts,amsthm,latexsym,graphicx,multirow,hyperref,enumerate,soul}

     \newcommand\Cay{\mathrm{Cay}}

\newcommand\Q{\mathrm{Q}}

\newtheorem{theorem}{Theorem}[section]
\newtheorem{lemma}[theorem]{Lemma}

\newtheorem{corollary}[theorem]{Corollary}

\theoremstyle{definition}
\newtheorem{question}[theorem]{Question}
\newtheorem{definition}[theorem]{Definition}

\newtheorem{construction}[theorem]{Construction}
\newtheorem{example}[theorem]{Example}
\newtheorem{remark}[theorem]{Remark}

\newcommand{\bmat}[1]{\begin{bmatrix}#1\end{bmatrix}}

\usepackage{xcolor}
\usepackage[normalem]{ulem}

\begin{document}
\openup 0.5\jot

\title[Regular sets in Cayley graphs]{Regular sets in Cayley graphs}

\author[Wang]{Yanpeng Wang}
\thanks{Corresponding author: Yanpeng Wang}
\address{Rongcheng Campus\\Harbin University of Science and Technology\\Harbin, Heilongjiang 150080\\People's Republic of China}
\email{wangyanpeng@pku.edu.cn}

\author[Xia]{Binzhou Xia}
\address{School of Mathematics and Statistics\\The University of Melbourne\\Parkville, VIC 3010\\Australia}
\email{binzhoux@unimelb.edu.au}

\author[Zhou]{Sanming Zhou}
\address{School of Mathematics and Statistics\\The University of Melbourne\\Parkville, VIC 3010\\Australia}
\email{sanming@unimelb.edu.au}
\maketitle
\begin{abstract}
In a graph $\Gamma$ with vertex set $V$, a subset $C$ of $V$ is called an $(a,b)$-regular set if every vertex in $C$ has exactly $a$ neighbors in $C$ and every vertex in $V\setminus C$ has exactly $b$ neighbors in $C$, where $a$ and $b$ are nonnegative integers. In the literature $(0,1)$-regular sets are known as perfect codes and $(1,1)$-regular sets are known as total perfect codes. In this paper we prove that, for any finite group $G$, if a non-trivial normal subgroup $H$ of $G$ is a perfect code in some Cayley graph of $G$, then for any pair of integers $a$ and $b$ with $0\leqslant a\leqslant|H|-1$ and $0\leqslant b\leqslant |H|$ such that $\gcd(2,|H|-1)$ divides $a$, $H$ is also an $(a,b)$-regular set in some Cayley graph of $G$ depending on $(a, b)$. A similar result involving total perfect codes is also proved in the paper.



\end{abstract}

\section{Introduction}

All groups considered in the paper are finite, and all graphs considered are finite and undirected with no loops or parallel edges. Let $\Gamma$ be a graph with vertex set $V$. For a vertex $v$ of $\Gamma$, denote by $\Gamma(v)$ the \emph{neighborhood} of $v$ in $\Gamma$, where two vertices are called \emph{neighbors} of each other if they are adjacent in the graph. A subset $C$ of $V$ is called a \emph{perfect code}~\cite{Kratochvil1986} in $\Gamma$ if every vertex of $\Gamma$ is at distance no more than one to exactly one vertex of $C$ (in particular, $C$ is an independent set of $\Gamma$).
A subset $C$ of $V$ is said to be a \emph{total perfect code}~\cite{Zhou2016} in $\Gamma$ if every vertex of $\Gamma$ has exactly one neighbor in $C$ (in particular, $C$ induces a matching in $\Gamma$ and so $|C|$ is even). In the literature a perfect code is also called an \emph{efficient dominating set}~\cite{DS2003} or \emph{independent perfect dominating set}~\cite{Lee2001}, and a total perfect code is also called an \emph{efficient open dominating set}~\cite{HHS1998}.


As a generalization of perfect and total perfect codes in a graph, the following notion was studied in, for example, \cite{BCGG2019, RCZ2018, Cardoso2019}.

\begin{definition}
\label{def:ab}
Let $\Gamma$ be a graph with vertex set $V$, and let $a$ and $b$ be nonnegative integers. A nonempty proper subset $C$ of $V$ is called an $(a,b)$-\emph{regular set} in $\Gamma$ if $|\Gamma(v)\cap C|=a$ for each $v\in C$ and $|\Gamma(v)\cap C|=b$ for each $v\in V\setminus C$. An $(a,b)$-regular set is simply called a \emph{regular set} if the parameters $a$ and $b$ are not important in the context.
\end{definition}

In particular, a $(0,1)$-regular set in $\Gamma$ is exactly a perfect code in $\Gamma$, and a $(1,1)$-regular set in $\Gamma$ is exactly a total perfect code in $\Gamma$.

It is not difficult to see that regular sets in a regular graph are precisely equitable partitions of the graph into two parts. In general, a partition $\mathcal{V} = \{V_1, V_2, \dots, V_r\}$ of the vertex set of a graph $\Gamma$ is said to be \emph{equitable} \cite[$\mathsection9.3$]{GR2001} if there exists an $r\times r$ matrix $M=(m_{ij})$ such that for any $i, j \in \{1, 2, \ldots, r\}$, every vertex in $V_i$ has exactly $m_{ij}$ neighbors in $V_j$.
The matrix $M$ is called the \emph{quotient matrix}~\cite{BCGG2019} of the partition $\mathcal{V}$.
If $\Gamma$ is a connected $k$-regular graph, then $M$ has all row sums equal to $k$, and so $k$ is a simple eigenvalue of $M$~\cite[Theorem 9.3.3]{GR2001}. The equitable partition $\mathcal{V}$ of $\Gamma$ is said to be \emph{$\mu$-equitable}~\cite{BCGG2019} if all eigenvalues of its quotient matrix $M$ other than $k$ are equal to $\mu$. It is shown in~\cite[Corollary~2.3]{BCGG2019} that a non-trivial coarsening of a $\mu$-equitable partition is $\mu$-equitable. Thus it is especially important to study equitable partitions with exactly two parts, and for regular graphs such partitions are precisely regular sets in the graph. In fact, it can be verified (see also \cite{RCZ2018}) that for a connected $k$-regular graph $\Gamma$ with vertex set $V$, a nonempty proper subset $C$ of $V$ is an $(a,b)$-regular set in $\Gamma$ if and only if $\{C,V\setminus C\}$ is a $\mu$-equitable partition of $\Gamma$, where $a, b$ and $\mu = a-b$ are related by
$$
a=((k-\mu)|C|+\mu|V|)/|V|
$$
and
$$
b=((k-\mu)|C|)/|V|.
$$
(This can be proved using the fact that the quotient matrix of any $\mu$-equitable partition of $\Gamma$ has trace $k+\mu$ and that $|C|(k-a)=(|V|-|C|)b$ for any $(a,b)$-regular set $C$ in $\Gamma$.) In particular, any regular graph $\Gamma$ admitting an $(a,b)$-regular set must have $a-b$ as an eigenvalue (see also \cite{RCZ2018}), because all eigenvalues of the quotient matrix of any equitable partition of $\Gamma$ are also eigenvalues of $\Gamma$ (see \cite[Theorem 9.3.3]{GR2001}). This generalizes the well-known result that any regular graph admitting a perfect code should have $-1$ as an eigenvalue (see \cite[Lemma 9.3.4]{GR2001}). From a coding theoretical point of view, an $(a, b)$-regular set in a $k$-regular graph $\Gamma$ is exactly a completely regular code $C$ in $\Gamma$ (see, for example, \cite{Neumaier}) such that the corresponding distance partition has exactly two parts, namely $\{C, V \setminus C\}$, with quotient matrix $\bmat{a & k-a\\ b & k-b}$.

Perfect codes in Cayley graphs have attracted special attention \cite{Dinitz2006,Hajos1942,HXZ2018,Szab2006,SS2009} since they are generalizations of perfect codes under the Hamming and Lee metrics and are closely related to factorizations and tilings of groups. Denote by $e$ the identity element of the group under consideration. For a group $G$ and an inverse-closed subset $S$ of $G\setminus\{e\}$, the \emph{Cayley graph} $\Cay(G,S)$ of $G$ with \emph{connection set} $S$ is defined to be the graph with vertex set $G$ such that $x,y\in G$ are adjacent if and only if $yx^{-1}\in S$. It was observed in \cite{HXZ2018} that subgroups of a given group which are perfect codes in some Cayley graphs of the group are particularly interesting since they are an analogue of perfect linear codes in the classical setting of coding theory. In general, if a subset $C$ of $G$ is a (total) perfect code in some Cayley graph of $G$, then $C$ is called a \emph{(total) perfect code of $G$} \cite{HXZ2018}. Subgroups which are also perfect codes of the group were studied in \cite{HXZ2018} and \cite{MWWZ2019}, and a characterization of those groups whose subgroups are all perfect codes of the group was given in \cite{MWWZ2019}.
In~\cite[Theorem~2.2]{HXZ2018}, it was proved that a normal subgroup $H$ of $G$ is a perfect code of $G$ if and only if
\begin{equation}\label{equ9}
\text{for any $g\in G$ with $g^2\in H$, there exists $h\in H$ such that $(gh)^2=e$},
\end{equation}
and that $H$ is a total perfect code of $G$ if and only if~\eqref{equ9} holds and $|H|$ is even.

Generalizing the concept of perfect codes of a group, we call a subset $C$ of a group $G$ an \emph{$(a,b)$-regular set of $G$} if $C$ is an $(a,b)$-regular set in some Cayley graph of $G$.
Thus a perfect code of $G$ is precisely a $(0,1)$-regular set of $G$, and a total perfect code of $G$ is precisely a $(1,1)$-regular set of $G$. In line with the study in \cite{HXZ2018}, it is natural to ask when a normal subgroup of a group is an $(a,b)$-regular set of the group. We answer this question in this paper by proving the following theorem which is the main result of the paper.

\begin{theorem}\label{thm5}
Let $G$ be a group and let $H$ be a non-trivial normal subgroup of $G$. Then the following statements are equivalent:
\begin{enumerate}
\item[{\rm(a1)}] $G$ and $H$ satisfy condition \eqref{equ9};
\item[{\rm(a2)}] $H$ is a perfect code of $G$;
\item[{\rm(a3)}] $H$ is an $(a,b)$-regular set of $G$ for every pair of integers $a$ and $b$ with $0\leqslant a\leqslant|H|-1$ and $0\leqslant b\leqslant |H|$ such that $\gcd(2,|H|-1)$ divides $a$.
\end{enumerate}
And the following statements are also equivalent:
\begin{enumerate}
\item[{\rm(b1)}] $G$ and $H$ satisfy condition \eqref{equ9}, and $|H|$ is even;
\item[{\rm(b2)}] $H$ is a total perfect code of $G$;
\item[{\rm(b3)}] $H$ is an $(a,b)$-regular set of $G$ for every pair of integers $a$ and $b$ with $0\leqslant a\leqslant|H|-1$ and $0\leqslant b\leqslant |H|$.
\end{enumerate}
\end{theorem}

The equivalence of (a1) and (a2) and that of (b1) and (b2) have been proved in \cite[Theorem~2.2]{HXZ2018}. So the essence of Theorem \ref{thm5} lies in that (a2) implies (a3) and (b2) implies (b3). Moreover, as will be seen in Construction \ref{con1}, based on an inverse-closed subset $S_0$ of $G\setminus\{e\}$ such that $\Cay(G,S_0)$ admits $H$ as a perfect code, we will give a construction of an inverse-closed subset $S$ of $G\setminus\{e\}$ such that $\Cay(G,S)$ admits $H$ as an $(a, b)$-regular set. A construction of an inverse-closed subset $S_0$ of $G\setminus\{e\}$ such that $\Cay(G,S_0)$ admits a given normal subgroup $H$ satisfying \eqref{equ9} as a perfect code was given in the proof of \cite[Theorem~2.2]{HXZ2018}. Combining this construction with Construction \ref{con1}, we can construct, for any non-trivial normal subgroup $H$ of $G$ satisfying \eqref{equ9} and every pair of integers $a$ and $b$ as in (a3), an inverse-closed subset $S$ of $G\setminus\{e\}$ depending on $(a,b)$ such that $\Cay(G,S)$ admits $H$ as an $(a, b)$-regular set.

The condition that $H$ is normal in $G$ will be used in our proof of Theorem~\ref{thm5}. However, we do not know any example of a non-normal subgroup $H$ of a group $G$ such that the equivalence of (a1) and (a2) or that of (b1) and (b2) fails. This prompts us to ask the following question.
\begin{question}
\label{que1}
Is it still true that (a1) and (a2) in Theorem \ref{thm5} are equivalent if the subgroup $H$ of $G$ is not normal? Is it still true that (b1) and (b2) in Theorem \ref{thm5} are equivalent if the subgroup $H$ of $G$ is not normal?
\end{question}



The rest of the paper is structured as follows. In the next section we will prove a lemma which will be used in the proof of Theorem \ref{thm5}. In Section~\ref{sec1}, we will establish some general results on subgroup perfect codes in Cayley graphs and prove Theorem~\ref{thm5} at the end of the section. We will conclude the paper with examples and remarks in Section~\ref{sec:remarks}.

\section{A lemma}

As usual, for a group $G$, denote by $\mathbb{Z}[G]$ the group ring of $G$ over $\mathbb{Z}$. For a subset $A$ of the group $G$, denote
\[
\overline{A}=\sum_{g\in G}\mu_A(g)g\in \mathbb{Z}[G],
\]
where
\[
\mu_A(g)=\left\{\begin{aligned}
1,&\quad g\in A;\\
0,&\quad g\in G\setminus A.
\end{aligned}
\right.
\]
In~\cite[Lemma~2.10]{HXZ2018}, a characterization of perfect codes and total perfect codes in Cayley graphs was given in the language of group rings. The following lemma extends this result to the general case of $(a, b)$-regular sets.

\begin{lemma}\label{thm1}
Let $G$ be a group, $C$ a subset of $G$, and $S$ an inverse-closed subset of $G\setminus\{e\}$. Let $a$ and $b$ be nonnegative integers. Then the following statements are equivalent:
\begin{enumerate}[{\rm(a)}]
\item $C$ is an $(a,b)$-regular set in $\Cay(G,S)$;
\item $|Sx\cap C|=a$ for each $x\in C$ and $|Sx\cap C|=b$ for each $x\in G\setminus C$;
\item $\overline{S}\cdot\overline{C}=a\,\overline{C}+b\,\overline{G\setminus C}$;
\item $\overline{S}\cdot\overline{C}+(b-a)\,\overline{C}=b\,\overline{G}$.
\end{enumerate}
\end{lemma}

\begin{proof}
It is clear that (a) and (b) are equivalent and (c) and (d) are equivalent. Since $S$ is inverse-closed, we have
\begin{align*}
\overline{S}\cdot\overline{C}
&=\sum_{s\in S}\sum_{c\in C}sc\\
&=\sum_{x\in G}\,\Bigg(\sum_{(s,c)\in S\times C, sc=x}1\Bigg)x \\
&=\sum_{x\in G}\Bigg(\sum_{c\in C, xc^{-1}\in S}1\Bigg)x \\
&=\sum_{x\in G}\Bigg(\sum_{c\in C, c\in S^{-1}x}1\Bigg)x \\
&=\sum_{x\in G}|S^{-1}x\cap C|x\\
&=\sum_{x\in G}|Sx\cap C|x\\
&=\sum_{x\in C}|Sx\cap C|x+\sum_{x\in G\setminus C}|Sx\cap C|x.
\end{align*}
Note that (b) holds if and only if
\[
\sum_{x\in C}|Sx\cap C|x=a\,\overline{C}
\]
and
\[
\quad\sum_{x\in G\setminus C}|Sx\cap C|x=b\,\overline{G\setminus C}.
\]
It follows that (b) and (c) are equivalent. This completes the proof.
\end{proof}

Since a $(0,1)$-regular set is precisely a perfect code, in the special case when $(a, b)=(0,1)$, Lemma~\ref{thm1} gives rise to the following known result.

\begin{corollary}\label{cor1}
\emph{(\cite[Lemma~2.10]{HXZ2018})}
Let $G$ be a group, $C$ a subset of $G$, and $S$ an inverse-closed subset of $G\setminus\{e\}$.
Then $C$ is a perfect code in $\Cay(G,S)$ if and only if $\overline{S\cup\{e\}}\cdot\overline{C}=\overline{G}$.
\end{corollary}

\section{Subgroup regular sets}
\label{sec1}

Whenever we use the notation $\sqcup$ we mean the sets involved in the union are pairwise disjoint. For example, $A \sqcup B$ is the union of disjoint sets $A$ and $B$, and $\sqcup_{i=1}^n A_i$ is the union of pairwise disjoint sets $A_1, A_2, \ldots, A_n$.

\begin{lemma}\label{lem1}
Let $G$ be a group, $H$ a subgroup of $G$, and $S$ an inverse-closed subset of $G\setminus\{e\}$. Let $a$ and $b$ be nonnegative integers. Then $H$ is an $(a,b)$-regular set in $\Cay(G,S)$ if and only if $|S\cap H|=a$ and $\overline{S\setminus H}\cdot \overline{H}=b\,\overline{G\setminus H}$.
\end{lemma}

\begin{proof}
According to Lemma~\ref{thm1}, $H$ is an $(a,b)$-regular set in $\Cay(G,S)$ if and only if $\overline{S}\cdot\overline{H}=a\,\overline{H}+b\,\overline{G\setminus H}$. Since $S=(S\cap H) \sqcup (S\setminus H)$ and $h \overline{H}=\overline{H}$ for all $h\in H$, we have
\begin{align*}
\overline{S}\cdot\overline{H}&=(\overline{S\cap H}+\overline{S\setminus H})\cdot\overline{H}\\
&=\overline{S\cap H}\cdot\overline{H}+\overline{S\setminus H}\cdot\overline{H}\\
&=|S\cap H|\overline{H}+\overline{S\setminus H}\cdot\overline{H}.
\end{align*}
Thus the result follows since $sH\cap H =\emptyset$ for all $s\in S\setminus H$.
\end{proof}

\begin{lemma}\label{lem4}
Let $G$ be a group, $H$ a subgroup of $G$, and $S$ an inverse-closed subset of $G\setminus\{e\}$. Let $a$ and $b$ be nonnegative integers. Suppose that $H$ is a perfect code in some Cayley graph $\Cay(G,S_0)$ of $G$. Then $H$ is an $(a,b)$-regular set in $\Cay(G,S)$ if and only if $|S\cap H|=a$ and $\overline{S\setminus H}\cdot \overline{H}=b\,\overline{S_0}\cdot \overline{H}$.
\end{lemma}

\begin{proof}
Since $H$ is a perfect code in $\Cay(G,S_0)$, we derive from Corollary \ref{cor1} that
\[
\overline{G}=\overline{S_0\cup\{e\}}\cdot\overline{H}=\overline{S_0}\cdot\overline{H}+\overline{H}.
\]
Hence
\[
\overline{G\setminus H}=\overline{G}-\overline{H}=\overline{S_0}\cdot\overline{H}.
\]
This together with Lemma~\ref{lem1} implies that $H$ is an $(a,b)$-regular set in $\Cay(G,S)$ if and only if
$|S\cap H|=a$ and $\overline{S\setminus H}\cdot\overline{H}=b\,\overline{S_0}\cdot\overline{H}$.
\end{proof}

Note that if a Cayley graph $\Cay(G,S_0 )$ admits a subgroup $H$ of $G$ as a perfect code, then $H$ and $S_0$ must be disjoint. With this in mind we give the following construction.

\begin{construction}\label{con1}
Given a group $G$, a normal subgroup $H$ of $G$, an inverse-closed subset $K$ of $H\setminus\{e\}$, a nonnegative integer $b\leqslant|H|$ and an inverse-closed subset $S_0$ of $G\setminus\{e\}$ such that $H$ is a perfect code in $\Cay(G,S_0)$, construct a subset $S$ of $G$ as follows.

Write
$$
H=\{h_1,h_2,\ldots,h_d\}
$$
and
$$
S_0=\{s_1,s_2,\ldots,s_{2m-1},s_{2m},s_{2m+1},\ldots,s_n\},
$$
where $d=|H|$, $n=|S_0|$, $s_i^{-1}=s_{2m+1-i}$ for $i\in \{1,2,\dots,2m\}$ and $s^{-1}_j=s_j$ for $j\in \{2m+1,2m+2,\dots,n\}$. Note that $s_{j}H$ is inverse-closed as $H$ is normal in $G$ and $s_j$ is an involution, and $e \notin s_{j}H$ as $H$ is a perfect code in $\Cay(G,S_0)$. Thus, for $j\in \{2m+1,2m+2,\dots,n\}$, we may write
\begin{equation*}\label{equ1}
s_jH=\{u_{j,1},u^{-1}_{j,1},u_{j,2},{u^{-1}_{j,2}},\ldots,u_{j,\alpha_j},{u^{-1}_{j,\alpha_j}},
v_{j,1},v_{j,2},\ldots,v_{j,\beta_j}\}
\end{equation*}
with the order of $u_{j,k}$ greater than $2$ for $k\in \{1,2,\ldots,\alpha_j\}$, $v_{j,\ell}$ an involution for $\ell\in\{1,2,\ldots,\beta_j\}$, and $v_{j,1}=s_j$. Since $2\alpha_j+\beta_j = |H| = d \ge b$, we have $\beta_j \ge b - 2\alpha_j$. For $i\in\{1,2,\ldots,b\}$, let
\begin{equation}\label{equ3}
S_i=\{s_1h_i,s_2h_i,\ldots,s_mh_i,(s_mh_i)^{-1},\ldots,(s_2h_i)^{-1},(s_1h_i)^{-1}\}.
\end{equation}
For $j\in\{2m+1, 2m+2, \ldots,n\}$, let
\[
T_j=
\begin{cases}
\big\{u_{j,1},u^{-1}_{j,1},u_{j,2},{u^{-1}_{j,2}},\ldots,u_{j,\alpha_j},{u^{-1}_{j,\alpha_j}},v_{j,1},v_{j,2},\ldots,v_{j,b-2\alpha_j}\big\}
&\text{if $b>2\alpha_j$;}\\
\big\{u_{j,1},u^{-1}_{j,1},u_{j,2},{u^{-1}_{j,2}},\ldots,u_{j,\frac{b-1}{2}},u^{-1}_{j,\frac{b-1}{2}},v_{j,1}\big\}
&\text{if $b\leqslant 2\alpha_j$ and $2\nmid b$;}\\
\big\{u_{j,1},u^{-1}_{j,1},u_{j,2},{u^{-1}_{j,2}},\ldots,u_{j,\frac{b}{2}},u^{-1}_{j,\frac{b}{2}}\big\}
&\text{if $b\leqslant 2\alpha_j$ and $2\mid b$.}
\end{cases}
\]
Let
\begin{equation}
\label{eq:SKT}
S=K\cup\left(\bigcup_{i=1}^bS_i\right)\cup\left(\bigcup_{j=2m+1}^nT_j\right).
\end{equation}
\end{construction}

\begin{theorem}\label{thm3}
In the notation of Construction~$\ref{con1}$, the following hold:
\begin{enumerate}[{\rm(a)}]
\item $|S_i|=2m$ for $i\in\{1,2,\ldots,b\}$;
\item $|T_j|=b$ for $j\in\{2m+1,2m+2,\ldots,n\}$;
\item
$
S=K\sqcup\left(\bigsqcup_{i=1}^bS_i\right)\sqcup\left(\bigsqcup_{j=2m+1}^nT_j\right)
$;
\item $S$ is an inverse-closed subset of $G\setminus\{e\}$;
\item $H$ is a $(|K|,b)$-regular set in $\Cay(G,S)$.
\end{enumerate}
\end{theorem}

\begin{proof}
It is clear from Construction~\ref{con1} that $|T_j|=b$ and $S_i$, $T_j$ and $K$ are all inverse-closed subsets of $G\setminus\{e\}$ for $i\in\{1,2,\ldots,b\}$ and $j\in\{2m+1,2m+2,\ldots,n\}$. Thus statement~(b) holds, and $S$ is an inverse-closed subset of $G\setminus\{e\}$, as statement~(d) asserts.

Since $H$ is a perfect code in $\Cay(G,S_0)$, Corollary~\ref{cor1} implies that $S_0\cup\{e\}$ is an inverse-closed left transversal of $H$ in $G$. For $r\in\{1,2,\dots,m\}$ and $i\in\{1,2,\ldots,b\}$, we have
\begin{equation}\label{equ8}
(s_rh_i)^{-1}=h_i^{-1}s^{-1}_r\in Hs_r^{-1}=s_r^{-1}H=s_{2m+1-r}H.
\end{equation}
Hence the elements $s_1h_i,s_2h_i,\ldots,s_mh_i,(s_mh_i)^{-1},\ldots,(s_2h_i)^{-1},(s_1h_i)^{-1}$ are in pairwise distinct left cosets $s_1H,s_2H,\ldots,s_{2m-1}H,s_{2m}H$. As a consequence we obtain that
\[
s_1h_i,s_2h_i,\ldots,s_mh_i,(s_mh_i)^{-1},\ldots,(s_2h_i)^{-1},(s_1h_i)^{-1}
\]
are pairwise distinct, which implies that $|S_i|=2m$, proving statement~(a). Moreover, for $i\in\{1,2,\ldots,b\}$, we have
\begin{equation}\label{equ2}
S_i\subseteq\bigcup_{r=1}^{2m}s_rH
\end{equation}
and
\begin{equation}\label{equ4}
S_i\cap(s_rH)=
\begin{cases}
\big\{s_rh_i\big\}&\text{for $r\in\{1,2,\ldots,m\}$;}\\
\big\{(s_{2m+1-r}h_i)^{-1}\big\}&\text{for $r\in\{m+1,m+2,\ldots,2m\}$.}
\end{cases}
\end{equation}
Note that $K\subseteq H$ and $T_j\subseteq s_jH$ for $j\in\{2m+1,2m+2,\ldots,n\}$. We derive from~\eqref{equ2} that
\begin{equation}\label{equ5}
S=K\sqcup\left(\bigcup_{i=1}^bS_i\right)\sqcup\left(\bigsqcup_{j=2m+1}^nT_j\right).
\end{equation}
If $x\in S_i\cap S_j$ for distinct $i,j$ in $\{1,2,\ldots,b\}$, then \eqref{equ2}
implies $x\in s_rH$ for some $r \in \{1, 2, \ldots, 2m\}$ and so \eqref{equ4} immediately leads to a contradiction. Thus $\bigcup_{i=1}^bS_i=\bigsqcup_{i=1}^bS_i$. This together with~\eqref{equ5} proves statement~(c).

For $i\in\{1,2,\ldots,b\}$, by \eqref{equ8} and the construction of $S_i$ we have
\begin{align*}
\overline{S_i}\cdot\overline{H}&=\left(\sum_{r=1}^ms_rh_i+\sum_{r=1}^m(s_rh_i)^{-1}\right)\cdot\overline{H}\\
&=\sum_{r=1}^m s_rh_i\overline{H}+\sum_{r=1}^m(s_rh_i)^{-1}\overline{H}\\
&=\sum_{r=1}^ms_r\overline{H}+\sum_{r=1}^ms_{2m+1-r}\overline{H}\\
&=\sum_{r=1}^{2m}s_r\overline{H}.
\end{align*}
Hence
\begin{equation}\label{equ6}
\sum_{i=1}^{b}\overline{S_i}\cdot\overline{H}=\sum_{i=1}^{b}\sum_{r=1}^{2m}s_r\overline{H}=b\sum_{r=1}^{2m}s_r\overline{H}.
\end{equation}
For $j\in\{2m+1,2m+2,\ldots,n\}$, we derive from the construction of $T_j$ that the elements of $T_j$ are all in $s_jH$, whence
\[
\overline{T_j}\cdot\overline{H}=|T_j|s_j\overline{H}=bs_j\overline{H}.
\]
It follows that
\begin{equation}\label{equ7}
\sum_{j=2m+1}^n\overline{T_j}\cdot\overline{H}=\sum_{j=2m+1}^n\left(bs_j\overline{H}\right)=b\sum_{j=2m+1}^ns_j\overline{H}.
\end{equation}
Since $S\cap H=K$, we deduce from statement~(c) that
\[
S\setminus H = \left(\bigsqcup_{i=1}^bS_i\right)\sqcup\left(\bigsqcup_{j=2m+1}^nT_j\right).
\]
This together with \eqref{equ6} and \eqref{equ7} shows that
\begin{align*}
\overline{S\setminus H}\cdot\overline{H}&=\left(\sum_{i=1}^b\overline{S_i}+\sum_{j=2m+1}^n\overline{T_j}\right)\overline{H}\\
&=\sum_{i=1}^b\overline{S_i}\cdot\overline{H}+\sum_{j=2m+1}^n\overline{T_j}\cdot\overline{H}\\
&=b\sum_{r=1}^{2m}s_r\overline{H}+b\sum_{j=2m+1}^{n}s_j\overline{H}\\
&=b\sum_{k=1}^{n}s_k\overline{H}\\
&=b\,\overline{S_0}\cdot\overline{H}.
\end{align*}
Thus, since $H$ is a perfect code in $\Cay(G,S_0)$ and $S\cap H=K$, Lemma~\ref{lem4} implies that $H$ is a $(|K|,b)$-regular set in $\Cay(G,S)$, as statement (e) asserts. This completes the proof.
\end{proof}

\begin{corollary}\label{thm2}
Let $G$ be a group and let $H$ be a normal subgroup of $G$. If $H$ is a perfect code of $G$, then $H$ is an $(a,b)$-regular set of $G$ for every pair of integers $a$ and $b$ with $0\leqslant a\leqslant|H|-1$ and $0\leqslant b\leqslant |H|$ such that $\gcd(2,|H|-1)$ divides $a$.
\end{corollary}

\begin{proof}
If $|H|$ is odd, then $H\setminus\{e\}$ is partitioned into pairs of elements that are inverses of each other, and so $H\setminus\{e\}$ has an inverse-closed subset of size $a$ for each even integer $0\leqslant a\leqslant|H|-1$. If $|H|$ is even, then there exists an involution in $H$, and so $H\setminus\{e\}$ has an inverse-closed subset of size $a$ for each integer $0\leqslant a\leqslant|H|-1$. Since by our assumption $H$ is a perfect code of $G$, we may take an inverse-closed subset $S_0$ of $G\setminus\{e\}$ such that $H$ is a perfect code in $\Cay(G,S_0)$. Let $a$ and $b$ be integers such that $0\leqslant a\leqslant|H|-1$, $0\leqslant b\leqslant |H|$ and $\gcd(2,|H|-1)$ divides $a$. Note that $a$ is even if $|H|$ is odd, as $\gcd(2,|H|-1)$ divides $a$. We conclude that there exists an inverse-closed subset $K$ of $H\setminus\{e\}$ with $|K|=a$. Now let $S$ be as in Construction~\ref{con1}. Then Theorem~\ref{thm3} ensures that $H$ is an $(a,b)$-regular set in $\Cay(G,S)$. This completes the proof.
\end{proof}

We are now in a position to prove Theorem~\ref{thm5}.

\begin{proof}[Proof of Theorem~$\ref{thm5}$]
The equivalence of (a1) and~(a2) has been proved in~\cite[Theorem~2.2]{HXZ2018}.
Corollary~\ref{thm2} shows that (a2) implies (a3). On the other hand, (a3) implies~(a2) since perfect codes are $(0,1)$-regular sets. Thus statements (a1), (a2) and (a3) are equivalent.

Again, the equivalence of (b1) and~(b2) has been proved in \cite[Theorem~2.2]{HXZ2018}. Suppose that~(b2) holds.  Then (a1) holds and $|H|$ is even. By the equivalence of~(a1), (a2) and (a3) as shown above, we then infer that~(a3) holds. As $|H|$ is even, we have $\gcd(2,|H|-1)=1$. Thus~(a3) leads to~(b3).
This shows that~(b2) implies~(b3). On the other hand, suppose that~(b3) holds.
Then in particular $H$ is a $(1,1)$-regular set of $G$. That is, $H$ is a total regular set of $G$.
Hence~(b3) implies~(b2). So (b2) and (b3) are equivalent. Therefore, statements~(b1), (b2) and (b3) are all equivalent, completing the proof.
\end{proof}

\section{Examples and remarks}
\label{sec:remarks}

We illustrate Construction~\ref{con1} by the following example. Recall that the generalized quaternion group $Q_{4n}$ of order $4n$ (where $n\geqslant 2$) is the group with presentation $\langle x,y\mid x^{2n} = e, y^2 = x^n, y^{-1}xy=x^{-1}\rangle$.
%

\begin{example}\label{exm1}
Let $G=\langle x,y\mid x^{10}=e,y^2=x^5,y^{-1}xy=x^{-1}\rangle$ be the generalized quaternion group of order $20$. Let $H=\langle x^2\rangle=\{e,x^2,x^{-2},x^4,x^{-4}\}$, $K=\{x^2,x^{-2}\}$ and $S_0=\{y,y^{-1},x^5\}$.
Then $H$ is a normal subgroup of $G$, $K$ is an inverse-closed subset of $H \setminus \{e\}$, and by Corollary~\ref{cor1}, $H$ is a perfect code in $\Cay(G,S_0)$. Using Construction~\ref{con1}, we now construct an inverse-closed subset $S$ of $G \setminus \{e\}$ such that $H$ is a $(2, 3)$-regular set in $\Cay(G,S)$.

Write $s_1=y$, $s_2=y^{-1}$ and $s_3=x^5$ so that $S_0=\{s_1,s_2,s_3\}$. By \eqref{equ3}, we have
\[
S_1=\{s_1,s_1^{-1}\}=\{y,y^{-1}\},
\]
\[
S_2=\{s_1x^2,(s_1x^2)^{-1}\}=\{x^8y,(x^8y)^{-1}\},
\]
and
\[
 S_3=\{s_1x^4,(s_1x^4)^{-1}\}=\{x^6y,(x^6y)^{-1}\}.
\]
Since $s_3H=x^5H=xH$, we have $s_3H=\{x,x^{-1},x^3,x^{-3},x^5\}$ and $T_3=\{x,x^{-1},x^5\}$. So \eqref{eq:SKT} yields
$$
S = K\cup(S_1\cup S_2)\cup T_3
= \{x^2,x^{-2},y,y^{-1},x^8y,(x^8y)^{-1},x^6y,(x^6y)^{-1},x,x^{-1},x^5\}.
$$
Finally, by Theorem~\ref{thm3}, $S$ is an inverse-closed subset of $G \setminus \{e\}$ and $H$ is a $(2,3)$-regular set in $\Cay(G,S)$.
\qed
\end{example}

It may happen that a normal subgroup of a group is an $(a,b)$-regular set of the group for some (but not all) pairs of integers $a$ and $b$ as in (a3) (respectively, (b3)) of Theorem ~\ref{thm3} but is not a perfect code (respectively, total perfect code) of the group. We illustrate this by the following two examples.

\begin{example}
Let $\Q_8=\{1,-1,i,-i,j,-j,k,-k\}$ be the quaternion group. By Lemma~\ref{lem1}, the normal subgroup $H=\langle i\rangle=\{1,-1,i,-i\}$ of $\Q_8$ is a $(1,2)$-regular set in $\Cay(\Q_8,\{-1,j,-j\})$ and a $(2,2)$-regular set in $\Cay(\Q_8,\{i,-i,j,-j\})$. However, using \cite[Theorem~2.2]{HXZ2018}, one can verify that $H$ is not a perfect code of $\Q_8$.
\end{example}

\begin{example}
Let $G=\langle x,y\mid x^8=e,y^2=x^4,y^{-1}xy=x^{-1}\rangle$ be the generalized quaternion group of order $16$. Then $H=\langle x^2\rangle$ is a normal subgroup of $G$. By Lemma~\ref{lem1}, we see that $H$ is a $(2,2)$-regular set in $\Cay(G,S)$, where $S = \{x^2,x^{-2},x,x^{-1},y,y^{-1},xy,(xy)^{-1}\}$. However, by~\cite[Theorem~2.2]{HXZ2018} we can show that $H$ is not a total perfect code of $G$.
\end{example}

\begin{remark}
Let $G$ be a group and $H$ a subgroup of $G$. The following statements are immediate corollaries of Lemma~\ref{lem1}:
\begin{itemize}
\item[\rm (a)] For every integer $a$ with $0 \leqslant a\leqslant |H|-1$ such that $\gcd(2,|H|-1)$ divides $a$, $H$ is an $(a,|H|)$-regular set of $G$;
\item[\rm (b)] if $H$ is an $(a,b)$-regular set of $G$ for some pair of integers $a$ and $b$, then it is also an $(a,|H|-b)$-regular set of $G$.
\end{itemize}
\end{remark}

In fact, we obtain (a) by replacing $S$ in Lemma~\ref{lem1} by the union of $G \setminus H$ and any inverse-closed subset of size $a$ in $H$. Similarly, letting $H$ be an $(a,b)$-regular set of $\Cay(G,S)$, we obtain (b) by replacing $S$ with $(S\cap H)\cup(G\setminus(S\cup H))$ in Lemma~\ref{lem1}.

\smallskip
\noindent\textsc{Acknowledgements.} We appreciate the four anonymous referees for their helpful comments. The first author gratefully acknowledges the financial support from China Scholarship Council (No.~201806010040). The third author was supported by the National Natural Science Foundation of China (No.~61771019) and the Research Grant Support Scheme of The University of Melbourne. The third author is grateful to Peter Cameron for introducing the concept of perfect sets to him and bringing \cite{BCGG2019} to his attention.

\smallskip
\noindent\textsc{Data availability statement.} Data sharing not applicable to this article as no datasets were generated or analysed during the current study.

\end{document}